\documentclass[12pt]{amsart}

\usepackage{amssymb,amsthm,amsmath}
\usepackage[numbers,sort&compress]{natbib}
\usepackage{color}
\usepackage{graphicx}
\usepackage{amsthm}
\usepackage{physics}
\usepackage{enumitem}
\usepackage{tikz}
\usepackage{amssymb,amsthm,amsmath, subcaption,blkarray}
\usepackage[numbers,sort&compress]{natbib}
\usepackage{color}

\usepackage{graphicx}
\usepackage{tikz}
\usetikzlibrary{shapes}
\usepackage{xparse}
\usetikzlibrary{calc,arrows.meta,positioning}

\title[B\'{e}zout's Bound]{Extrema  of spectral band functions of  two dimensional discrete periodic Schr\"odinger operators}
\author[M. Faust]{Matthew Faust}
\address[M. Faust]{Department of Mathematics, Texas A\&M University, College Station, TX 77843-3368, USA} 
\address{Current address: Department of Mathematics, Michigan State University, East Lansing, MI 48824, USA} \email{mfaust@msu.edu}
\author[W. Liu]{Wencai Liu}
\address[W. Liu]{Department of Mathematics, Texas A\&M University, College Station, TX 77843-3368, USA} \email{liuwencai1226@gmail.com; wencail@tamu.edu}
\author[E. Luo]{Ethan Luo}
\address[E. Luo]{Department of Mathematics, Texas A\&M University, College Station, TX 77843-3368, USA} \email{el0337@tamu.edu}



\theoremstyle{plain}
\newtheorem{theorem}{Theorem}[section]
\newcommand{\R}{\mathbb{R}}
\newtheorem{corollary}[theorem]{Corollary}
\newtheorem{lemma}[theorem]{Lemma}
\newtheorem{proposition}[theorem]{Proposition}
\newtheorem{remark}{Remark}
\newcommand{\C}{\mathbb{C}}
\newcommand{\T}{\mathbb{T}}
\newcommand{\Z}{\mathbb{Z}}

\theoremstyle{plain}
\newtheorem{definition}{Definition}
\newtheorem{conjecture}{Conjecture}

\begin{document}	
% \newcommand{\N}{\mathbb{N}}
%-------------------------------------------
\begin{abstract}
We use B\'{e}zout's theorem and Bernstein-Khovanskii-Kushnirenko theorem to analyze the level sets of the extrema of the spectral band functions of discrete periodic Schr\"odinger operators on $\mathbb{Z}^2$.  These approaches improve upon previous results of Liu and Filonov-Kachkovskiy.
\end{abstract}
\maketitle 

%-------------------------------------------
\section{Introduction and main results}
The discrete periodic Schr\"odinger operator is a Laplacian $\Delta$ together with a periodic potential $V : \Z^2 \to \C$, often denoted as $\Delta + V$.  
We say that $V : \Z^2 \rightarrow \C$ is $(q_1,q_2)$-periodic if, for  all $n = (n_1, n_2) \in \Z^2$, 
\begin{equation}
    V(n) = V(n_1 + q_1, n_2) = V(n_1, n_2 + q_2).
\end{equation}
In this paper,  we always fix the period $q = (q_1, q_2) \in \Z^2$.

Given a function $u $ on $\Z^2$, $\Delta + V$ acts on $u$ as follows: for each $n = (n_1,n_2) \in \Z^2$,
\begin{equation}
    ((\Delta + V) u)(n)= u(n_1 + 1, n_2) + u(n_1 - 1, n_2) + u(n_1, n_2 + 1) + u(n_1, n_2 - 1) + V(n)u(n).
\end{equation}

A function $u : \Z^2 \to \C$ satisfies Floquet-Bloch boundary conditions for some $k = (k_1, k_2) \in \R^2$ if both of the following equations are satisfied for all $n = (n_1, n_2) \in \Z^2$:
\begin{equation}\label{g1}
    u(n_1 + q_1, n_2) = e^{2\pi i k_1}u(n),
\end{equation}
and
\begin{equation}\label{g2}
    u(n_1, n_2 + q_2) = e^{2\pi i k_2}u(n).
\end{equation}

 Let $Q = q_1q_2$. If $u$ satisfies \eqref{g1} and \eqref{g2}, then we can represent $u$ as the vector of $Q$ values $\{u(n_1,n_2) : n_1 \in [q_1], n_2 \in [q_2]\}$, where $[q_j]=\{1,2,\cdots,q_j\}$, $j=1,2$.

The periodic operator $\Delta+V$ with boundary conditions \eqref{g1} and \eqref{g2} can be realized as a linear operator (matrix) on the vector space $\{u(n_1,n_2): n_1\in[q_1], n_2\in[q_2]\}$. Denote the linear operator by $D_V(k)$, which depends on $V$ and $k=(k_1,k_2)$.
 More precisely,  $D_V(k)$ has the following expression for $q_1\geq 3$ and $ q_2\geq 3$ (see \cite{GKTBook,liu1,kuc2016}):
 \[ (D_V(k_1,k_2))_{(m,n),(m',n')} = \begin{cases} 
          V(m,n) & m=m', n = n' \\
          1 & i = (m-m')^2 + (n-n')^2 = 1 \\
          e^{2 \pi i k_1} & m'=1, m=q_1, n=n' \\ 
          e^{-2 \pi i k_1} & m'=q_1, m=1, n=n' \\ 
          e^{2 \pi i k_2} & m=m', n'=1, n=q_2 \\ 
          e^{-2 \pi i k_2} & m=m', n'=q_2, n=1 \\ 
          0 & \text{otherwise.}
       \end{cases}\]

For example, when $q_1 = 3$ and $q_2 = 3$ we get the matrix,

\[
\begin{blockarray}{cccccccccc}
(1,1) & (1,2) & (1,3) & (2,1) & (2,2) & (2,3) & (3,1) & (3,2) & (3,3) \\
    \begin{block}{(ccccccccc)c}
        V_{1,1} & 1 & e^{-2 \pi i k_2} & 1 & 0 & 0 & e^{-2 \pi i k_1} & 0 & 0 & (1,1)\\
        1 & V_{1,2} & 1 & 0 & 1 & 0 & 0 & e^{-2 \pi i k_1} & 0 & (1,2)\\
        e^{2 \pi i k_2} & 1 & V_{1,3} & 0 & 0 & 1 & 0 & 0 & e^{-2 \pi i k_1} & (1,3)\\
        1 & 0 & 0 & V_{2,1} & 1 & e^{-2 \pi i k_2} & 1 & 0 & 0 & (2,1)\\
        0 & 1 & 0 & 1 & V_{2,2} & 1 & 0 & 1 & 0 & (2,2)\\
        0 & 0 & 1 & e^{2 \pi i k_2} & 1 & V_{2,3} & 0 & 0 & 1 & (2,3)\\
        e^{2 \pi i k_1} & 0 & 0 & 1 & 0 & 0 & V_{3,1} & 1 & e^{-2 \pi i k_2} & (3,1)\\
        0 & e^{2 \pi i k_1} & 0 & 0 & 1 & 0 & 1 & V_{3,2} & 1 & (3,2)\\
        0 & 0 & e^{2 \pi i k_1} & 0 & 0 & 1 & e^{2 \pi i k_2} & 1 & V_{3,3} & (3,3)\\
    \end{block}
\end{blockarray}.
\]

 Floquet theory  says that the discrete periodic Schr\"odinger operator 
 $\Delta + V$  is unitarily equivalent  to the direct integral of $D_V(k)$ over $k\in (\R/\Z)^2$ (e.g., \cite{liu1,kuc2016,GKTBook}). Consequently, analyzing $D_V(k) $ is fundamental to understanding the spectral theory of $\Delta + V$. Therefore, the primary objective of this paper is to  study $D_V(k)$.

%By collecting the eigenvalues of $D_V(k)$ as $k$ varies over $\R^2$ we recover the spectrum of $\Delta + V$. Notice that as $e^{\pm 2 \pi i k_j}$ is periodic, we can reduce the domain of $k$ from $\R^2$ to $[0,1)^2$.
%It is easy to see that for a fixed $q \in \Z^2$, $D_V(k)$ has $Q$ eigenvalues
For each $k \in [0,1)^2$, define $\lambda_j(k)$ to be the function that returns the $j$th smallest eigenvalue of $D_V(k)$. Notice that, by definition, for each $k$ we have $\lambda_1(k) \leq \lambda_2(k) \leq \cdots \leq \lambda_{Q}(k)$. We call $\lambda_j(k)$ the $j$th spectral band function. 

%\begin{definition}
%{\color{blue} We can use $\lambda^V_j$ to denote the spectral band function of $D_V(k)$. For simplicity, use $\lambda_j$}
%    For any square $n \times n$ matrix $A(k)$, for all $1 \leq i \leq n$, the spectral band function $\lambda_i(k)$ represents the value of the $i$-th smallest eigenvalue of $A(k)$.
%\end{definition}

%A point of the band function is an extrema of $\lambda_i(k)$ if it is globally a minimum of maximum of $\lambda_i(k)$. %The structure of the extrema of the band functions play a consequential role in many problems such as Green's function asymptotics, Louisville type theorems, Anderson localization, and effective mass
 
About 30 to 40 years ago, significant progress was made in studying discrete periodic Schrödinger operators, particularly regarding the irreducibility of Bloch and Fermi varieties~\cite{GKTBook,ktcmh90,bktcm91,bat1,batcmh92,ls,battig1988toroidal}, as well as inverse problems~\cite{kapiii,Kapi,Kapii}.
Recently, there has been a significant resurgence of interest  in investigating the spectral theory of discrete periodic operators. These studies have addressed various topics, including the irreducibility of Bloch and Fermi varieties~\cite{fls,shjst20, flm22, liu1, flm23, fg}, flat bands~\cite{sy23}, extrema of the Bloch variety~\cite{FS, fk18, fk2,dksjmp20}, Fermi and Floquet isospectrality~\cite{liu2d, liujde, flmrp}, quantum ergodicity~\cite{ms22, liu2022bloch}, Borg’s theorem~\cite{liuborg}, the Bethe-Sommerfeld conjecture~\cite{ef, hj18, fh}, and embedded eigenvalues~\cite{liu1,shi1,kv06cmp,kvcpde20,lmt}. For a comprehensive overview of the importance and historical context of these studies, see~\cite{GKToverview, col91, kuc2016, kusu01, kuchment2023analytic, liujmp22}.

In this paper, we concern ourselves with the extrema of the spectral band functions $\lambda_j(k)$, $j=1,2,\cdots,Q$. 
%A well known and widely believed conjecture is the spectral edges conjecture~\cite[Conjecture 3]{liu1}. This conjecture states that, for a generic potential, the extrema of the band functions are (1) attained by a single band, (2) are isolated, and (3) are nondegenerate. Part (1) of this conjecture was proven in ~\cite{kr}. Part (2) of this conjecture was proven in ~\cite{liu1} when $q_1$ and $q_2$ are coprime, and in \cite{fk2} it was shown that there exist periods such that $(2)$ is false. Some progress on part (3) of this conjecture has been investigated~\cite{ks87, col91}.
%In~\cite{liu1}, as part of proving part (2) of the spectral edges conjecture, it is shown that when $q_1$ and $q_2$ are coprime, the collection of extrema for any given band is a finite set of at most $4(q_1+q_2)^2$ points. 
In this work, we build upon the general strategy outlined in~\cite{liu1}. A significant result from~\cite{liu1} demonstrates that the cardinality of each level set of extrema can be bounded by analyzing a system of two Laurent polynomials in two variables. Based on this fact, one of the authors in~\cite{liu1} employed B\'{e}zout's theorem to obtain a bound on the cardinality of each level set of extrema.

In this paper, we first refine this approach by applying an improved version of B\'{e}zout's theorem. Specifically, we introduce an algebraic change of variables, leading to a better bound on the cardinality of each level set of the extrema. This bound improves upon the results in~\cite{liu1}. Furthermore, we utilize the Bernstein-Khovanskii–Kushnirenko (BKK) theorem, which often provides sharper bounds than B\'{e}zout's theorem, to further constrain the number of the extrema in each level set.

As a result, we establish a tighter bound than that obtained from the (improved) application of B\'{e}zout's theorem. In particular, by applying the BKK theorem, we show that if \(\lambda_*\) corresponds to an extremum of \(\lambda_m(k)\), and \(q_1\) and \(q_2\) are coprime, then the set  
\[
\{k \in [0,1)^2 : \lambda_m(k) = \lambda_*\}
\]  
has cardinality at most \(4q_1q_2\).

%-------------------------------------------
\section{Useful properties}

Recall from the introduction, we defined $D_V(k)$ to be the $Q \times Q$ matrix representing the Floquet transform of the discrete periodic Schr\"odinger operator with Floquet-Bloch boundary conditions (3) and (4). Let $z_j=e^{2\pi k_j}$ for $j=1,2$, and let $\T$ be the complex unit circle.
Clearly, we have a bijection between $z \in \T^2$ and $k \in [0,1)^2$. Thus, we may express $D_V(k)$ as $\mathcal{D}_V(z)$, while just changing the domain. For example, when $q = (3,3)$ we obtain the matrix, 

\[
\mathcal{D}_V(z) =
\begin{blockarray}{cccccccccc}
(1,1) & (1,2) & (1,3) & (2,1) & (2,2) & (2,3) & (3,1) & (3,2) & (3,3) \\
    \begin{block}{(ccccccccc)c}
        V_{1,1} & 1 & z_2^{-1} & 1 & 0 & 0 & z_1^{-1} & 0 & 0 & (1,1)\\
        1 & V_{1,2} & 1 & 0 & 1 & 0 & 0 & z_1^{-1} & 0 & (1,2)\\
        z_2 & 1 & V_{1,3} & 0 & 0 & 1 & 0 & 0 & z_1^{-1} & (1,3)\\
        1 & 0 & 0 & V_{2,1} & 1 & z_2^{-1} & 1 & 0 & 0 & (2,1)\\
        0 & 1 & 0 & 1 & V_{2,2} & 1 & 0 & 1 & 0 & (2,2)\\
        0 & 0 & 1 & z_2 & 1 & V_{2,3} & 0 & 0 & 1 & (2,3)\\
        z_1 & 0 & 0 & 1 & 0 & 0 & V_{3,1} & 1 & z_2^{-1} & (3,1)\\
        0 & z_1 & 0 & 0 & 1 & 0 & 1 & V_{3,2} & 1 & (3,2)\\
        0 & 0 & z_1 & 0 & 0 & 1 & z_2 & 1 & V_{3,3} & (3,3)\\
    \end{block}
\end{blockarray}.
\]

Let $P_V(k,\lambda)$ be the characteristic polynomial of $D_V(k)$, namely $P_V(k, \lambda) = \det(D_V(k) - \lambda I)$. Similarly, let $\mathcal{P}_V(z,\lambda)$ be the characteristic polynomial of $\mathcal{D}_V(z)$ ( $\mathcal{P}_V(z, \lambda) = \det(\mathcal{D}_V(z) - \lambda I)$). For the remainder of the paper, we assume $V$ to be a real valued $q$-periodic potential, and we drop the dependence of $V$ in our notation; for example, $\mathcal{P}(z,\lambda) $. Notice that $\mathcal{P}(z,\lambda)$ is a Laurent polynomial in the $z$ variables.
\begin{proposition}\label{prop}
Here we describe some of the known properties of $\mathcal{P}(z,\lambda)$ from the literature~\cite{liu1}. Assume without loss of generality that $q_1 \geq q_2$, then the following hold:
\begin{enumerate}
    \item For any $c_1, c_2 \in \mathbb{Z}$, such that $\mathcal{P}(z,\lambda_*)$ contains a non-zero coefficient on the term $z_1^{c_1}z_2^{c_2}$, we have 
    \[
    |c_1q_1| + |c_2q_2| \leq q_1q_2.
    \]
%    Furthermore, equality holds only when $c_1 = 0$ or $c_2 = 0$.
   \item   For a fixed $\lambda_*$, $\mathcal{P}(z,\lambda_*)$ is a polynomial of $z_1$, $z_1^{-1}$, $z_2$ and  $z_2^{-1}$  
   with terms $z_1^{q_2}$, $(z_1^{-1})^{q_2}$, $z_2^{q_1}$, and $(z_2^{-1})^{q_1}$ (up to a change of sign).
\end{enumerate}
\end{proposition}
From the Laurent polynomial $\mathcal{P}(z,\lambda_*)$, we can define a polynomial $\mathcal{P}_1(z,\lambda) = z_1^{q_2}z_2^{q_1}\mathcal{P}(z,\lambda)$.
\begin{theorem}\label{thmextrem}~\cite[Theorem 2.5]{liu1}
    Let $\lambda_{*}$ be an extremum of a band function $\lambda_m(k)$, for some $m \in [Q]$.
    Then we have 
    \begin{equation*}\label{gextrem}
		\{k\in \R^2: \lambda_m(k)=\lambda_{*}\} \subseteq \{k\in\R^2: {P}(k,\lambda_{*})=0, |\nabla_k {P}(k,\lambda_{*})|=0\},
    \end{equation*}
    where $\nabla$ is the gradient.
\end{theorem}

Letting $\C^{\star}=\C\backslash\{0\}$, we have the following lemma.
\begin{lemma}
    Let \begin{equation}
        S_1=\{z\in\C^2: \mathcal{P}_1(z,\lambda_*)=\nabla_z\mathcal{P}_1(z,\lambda_*)=0\},
    \end{equation} \begin{equation}
        S_2=\{z\in(\C^{\star})^2: \mathcal{P}(z,\lambda_*)=\nabla_z \mathcal{P}(z,\lambda_*)=0\},  
    \end{equation}
    and
    \begin{equation}
        S_3=\{z=(z_1,z_2): |z_1|=|z_2|=1,\mathcal{P}(z,\lambda_*)=\nabla_z \mathcal{P}(z,\lambda_*)=0\}.
    \end{equation}
    Then 
    \begin{equation}
       \# \{k\in[0,1)^2:  \lambda_{{m}}(k)=\lambda_{*} \}\leq \# S_3 \leq \#S_2 \leq \#S_1.
    \end{equation}
\end{lemma}
\begin{proof}
The proof follows from Theorem \ref{thmextrem} immediately.
    
\end{proof}

\begin{definition}
    A polynomial $f$ is said to be square free if there does not exist an irreducible polynomial $g$ such that $g^2$ divides $f$. 
\end{definition}

From here on, we always assume that $q_1$ and $q_2$ are coprime.

\begin{theorem}~\cite[Theorem 2.3 and Remark 4]{liu1}
For any $\lambda\in\C$, the polynomial $\mathcal{P}_1(z,\lambda)$ (as a function of $z$) is square free. Moreover, if it factors, then no factors are univariate.
\end{theorem}

\begin{lemma}~\label{lem:2}
    Suppose $f$ has no univariate factors, then $f$ is square free if and only if $f$ and $\frac{\partial}{\partial z_1} f$ share no common factors.
\end{lemma}
\begin{proof} 
It is clear that $f$ and $\frac{\partial}{\partial z_1} f$ sharing no common factors implies that $f$ is square free.

We prove the other direction by contradiction. Suppose $f$ is square free and assume that there exists an irreducible polynomial $g$ that divides both $f$ and $\frac{\partial}{\partial z_1} f$. As $g$ divides $f$, there exists a polynomial $h$ such that $f = gh$. Notice that $\frac{\partial}{\partial z_1} f = (\frac{\partial}{\partial z_1} g)h + (\frac{\partial}{\partial z_1} h) g$. By our assumption on $f$ having no univariate factors, we must have that both $\frac{\partial}{\partial z_1} g$ and $\frac{\partial}{\partial z_1} h$ are nonzero. 

Clearly $g$ divides $(\frac{\partial}{\partial z_1} h) g$, and thus $g$ divides $\frac{\partial}{\partial z_1} f$ if and only if $g$ divides $(\frac{\partial}{\partial z_1} g)h$. As $g$ is irreducible, $\gcd(g, (\frac{\partial}{\partial z_1} g)) = 1$, and so $g$ divides $(\frac{\partial}{\partial z_1} g)h$ if and only $g$ divides $h$. As we assumed that $f$ and $\frac{\partial}{\partial z_1} f$ share a common factor, it must be the case that $g$ divides $h$, but this then contradicts the assumption that $f$ is square free.
\end{proof}
% Introduce previous theorem
\begin{theorem}\label{thmex}~\cite[Theorem 1.4]{liu1}
    Let $\lambda_{*}$ be an extremum of $\lambda_m(k)$, $k\in [0,1)^2,  m=1,2,\cdots,Q$. Then the level set
    \begin{equation*}\label{last18}
		\{k\in[0,1)^2:  \lambda_{{m}}(k)=\lambda_{*} \}
    \end{equation*}
    has cardinality at most $4(q_1+q_2)^2$. 
\end{theorem}
\begin{remark}
Although $4(q_1+q_2)^2$ is obtained, by checking the details of the proof of \cite[Theorem 2.6]{liu1}, we could obtain the bound $(2q_1+q_2)(2q_1+q_2-1)$. This proof has been added in the appendix. Keep in mind that we are assuming that $q_1 \geq q_2$.
\end{remark}
%-------------------------------------------

\section{An Improved B\'{e}zout's Bound}~\label{SEC:3}
\begin{theorem}\label{thmex}
    Let $\lambda_{*}$ be an extremum  of   $\lambda_m(k)$, $k\in [0,1)^2,  m \in [1,Q]$. Then the level set
    \begin{equation*}\label{last18}
		\{k\in[0,1)^2:  \lambda_{{m}}(k)=\lambda_{*} \}
    \end{equation*}
    has cardinality at most $9q_1q_2-3$.
\end{theorem}

% Bezouts
\begin{theorem}[B\'{e}zout's Theorem] 
Let \( f(z_1, z_2) \) and \( g(z_1, z_2) \) be polynomials in two variables \( z_1 \) and \( z_2 \), with degrees \( d_1 \) and \( d_2 \), respectively. If \( f(z_1, z_2) \) and \( g(z_1, z_2) \) do not share a common factor, then the system \( f(z_1,z_2) = g(z_1,z_2) = 0 \) has at most \( d_1 d_2 \) solutions in $\C^2$, counting multiplicities.
\end{theorem}

%Before using the theorem, we need to verify that there are no components in common between the two polynomials.

% Proof of 3.1
\begin{proof}[\bf Proof of Theorem \ref{thmex}]
    We would like to give an upper bound on $\#S_1$ similar to the way we did for Theorem 2.5. We start by supposing that $z = (x_1^{q_1}, x_2^{q_2})$. We want to find the number of $x \in \C^2$ such that the following two equations hold,
\begin{equation}
    \mathcal{P}_1((x_1^{q_1}, x_2^{q_2}), \lambda_*) = 0,
\end{equation}
and
\begin{equation}
    \frac{\partial}{\partial x_1}\mathcal{P}_1((x_1^{q_1}, x_2^{q_2}), \lambda_*) = 0.
\end{equation}
By Proposition 2.1,
we  have that the degree of $\mathcal{P}_1((x_1^{q_1}, x_2^{q_2}), \lambda_*)$ in terms of $x_1$ and $x_2$, is $3q_1q_2$,
%This is because by Proposition 2.1, $\mathcal{P}((x_1^{q_1}, x_2^{q_2}), \lambda_*)$ has highest degree terms (up to a change of sign) $x_1^{q_1q_2}$, $x_1^{-q_1q_2}$, $x_2^{q_1q_2}$, and  $x_2^{-q_1q_2}$ and we are multiplying by $z_1^{q_2}z_2^{q_1} = x_1^{q_1q_2}x_2^{q_1q_2}$ to get to $\mathcal{P}_1((x_1^{q_1}, x_2^{q_2}), \lambda_*)$. 
and the degree of $\frac{\partial}{\partial x_1}\mathcal{P}_1((x_1^{q_1}, x_2^{q_2}), \lambda_*)$ is $3q_1q_2 - 1$. We can also see from Theorem 2.4 and Lemma 2.5 that $\mathcal{P}_1((x_1^{q_1}, x_2^{q_2}), \lambda_*)$ does not share a common factor with $\frac{\partial}{\partial x_1}\mathcal{P}_1((x_1^{q_1}, x_2^{q_2}), \lambda_*)$. Using B\'{e}zout's Theorem, we get that the number of solutions is at most $9q_1^2q_2^2 - 3q_1q_2$. We also know that for each solution in terms of $z$, there exist $q_1q_2$ solutions in terms of $x$. Dividing by $q_1q_2$, we get that $\#S_1$ is at most $9q_1q_2 - 3$. Applying Lemma 2.3 , we conclude that $\#\{k\in[0,1)^2:  \lambda_{{m}}(k)=\lambda_{*} \} \leq  9q_1q_2 - 3$.
\end{proof} 
%-------------------------------------------

\section{Bernstein-Khovanskii–Kushnirenko Bound}~\label{SEC:4}

\begin{theorem}\label{thmex}
    Let $\lambda_{*}$ be an extremum  of   $\lambda_m(k)$, $k\in [0,1)^2,  m=1,2,\cdots,Q$. Then the level set
    \begin{equation*}\label{last18}
		\{k\in[0,1)^2:  \lambda_{{m}}(k)=\lambda_{*} \}
    \end{equation*}
    has cardinality at most $4q_1q_2$.
\end{theorem}
The following definitions are standard in discrete geometry, see~\cite{Ewald}.
\begin{definition}
    For a Laurent polynomial $f(z)$, where $z$ is an $n$-dimensional variable, the support of $f(z)$ is the set of $s$ exponents $\mathcal{A}(f) = \{e_1,\dots, e_s\}$, where for every $1 \leq i \leq s$ $e_i \in \Z^n$, and there exists $c_i \neq 0$, such that $f(z) = \sum_{i = 1}^{s} c_i z^{e_i}$.
\end{definition}

For $n = 2$, $\mathcal{A}(f)$ represents the set of all pairs of exponents of $f$, which have non-zero coefficients. %This is the case that we will be using in the Theorem.

\begin{definition}
    The Newton polytope of $f(z)$, denoted by $N(f)$ is the convex hull of $\mathcal{A}(f)$ in $\R^n$.
\end{definition}

In particular, the Newton polytope of $f(z)$ represents the smallest convex set containing $\mathcal{A}(f)$. Given a polytope $P \subset \R^n$, we define $V(P)$ to denote the $n$-dimensional Euclidean volume. For $n = 2$, we can just take the area of the polygon formed in the plane.

\begin{definition}
    Given two polytopes $P_1$ and $P_2$ in $\R^2$ such that neither is contained in a single line in $\R^2$, their mixed volume\footnote{Note that mixed volume may also be defined as $2MV(P_1, P_2) = V(P_1 + P_2) - V(P_1) - V(P_2)$. We follow the convention of~\cite{Bernstein}.} is defined as \[MV(P_1, P_2) = V(P_1 + P_2) - V(P_1) - V(P_2). \]
\end{definition}
Here we take the addition of two polytopes to be their Minkowski sum, which is defined to be $P_1+P_2=\{p_1 + p_2 \mid p_1 \in P_1, p_2\in P_2\}$. We note that mixed volume is monotonic; that is, if $P_3$ is a polytope such that $P_2 \subseteq P_3$, then $MV(P_1,P_2) \leq MV(P_1,P_3)$.  

\begin{lemma}[see  ~\cite{Bernstein}]
    If $P$ is a $2$-dimensional polytope, then $MV(P,P) = 2V(P)$.
\end{lemma}

\begin{theorem}[Bernstein–Khovanskii-Kushnirenko theorem ~\cite{Bernstein}]~\label{thm:Bern}
Given a system of Laurent polynomial equations $f_1(z) = f_2(z) = 0$, where $z$ is a $2$-dimensional variable and $f_1(z)$ and $f_2(z)$ share no common factors, the number of solutions in $(\C^{\star})^2$, counted with multiplicity, is at most the mixed volume of the Newton polytopes of $f_1(z)$ and $f_2(z)$. 
\end{theorem}

\begin{lemma}~\label{lem:newpoly}
    The Newton polytope, $N$, of $\mathcal{P}(z,\lambda_*)$ is that of Figure~\ref{polytope}.
    \begin{figure}[h!]
    \centering
    \begin{tikzpicture}
        \draw [thin, gray, ->] (0,-2.5) -- (0,2.5)
            node [above, black] {$z_2$};
        
        \draw [thin, gray, ->] (-2.5,0) -- (2.5,0)
            node [right, black] {$z_1$}; 

        \filldraw[thick,fill=gray,fill opacity=0.4] (0,2) -- (1,0) -- (0,-2) -- (-1,0) -- cycle; 
        \node [left] at (0,2) {$q_1$};
        \node [below] at (1.12,0) {$q_2$};
        \node [left] at (0,-2) {$-q_1$};
        \node [below] at (-1.3,0) {$-q_2$};
    \end{tikzpicture}
    \caption{$N$, a convex area containing all of the exponents of $\mathcal{P}(z,\lambda_{*})$.}
      \label{polytope}
\end{figure}
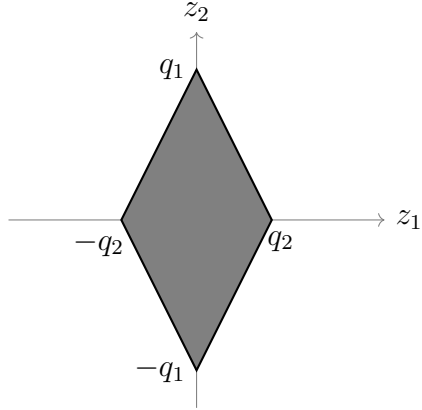
\end{lemma}
\begin{proof}
This follows immediately from Proposition 2.1 and the existence of $(q_2, 0)$, $(0, q_1)$, $(-q_2, 0)$, and $(0, -q_1)$ as exponents of $\mathcal{P}(z,\lambda_{*})$.
\end{proof}

\begin{corollary}~\label{cor:newpoly}
$N$ contains all the exponent vectors of $z_1(\frac{\partial}{\partial z_1}\mathcal{P}(z,\lambda_{*}))$.
\end{corollary}
\begin{proof}
    Taking the partial derivative of $\mathcal{P}(z,\lambda_{*})$ with respect to $z_1$ is the same as taking the sum of the partial derivatives of each of its terms with respect to $z_1$. These terms either vanish individually or have their corresponding powers of $z_1$ decreased by $1$.  As we multiply this result by $z_1$, we see that $N$ is a convex area that contains all the exponent vectors of $z_1(\frac{\partial}{\partial z_1}\mathcal{P}(z,\lambda_{*}))$.
\end{proof}

\begin{proof}[\bf Proof of Theorem 4.1]
    By Lemma 2.3, we know that $\#\{k \in [0,1)^2 : \lambda_m(k) = \lambda_*\} \leq \#S_2$. Any solution that satisfies all three of the equations in $S_2$ must also satisfy any pair of them. Thus, we can bound our answer by the number of solutions to the pair of equations,

\begin{equation}~\label{eq:1}
    \mathcal{P}(z,\lambda_*) = 0,
\end{equation}
and
\begin{equation}~\label{eq:2}
    \frac{\partial}{\partial z_1}\mathcal{P}(z,\lambda_*) = 0.
\end{equation}

Define $S_4$ to be the set of solutions $z \in (\C^{\star})^2$ that satisfy both \eqref{eq:1} and \eqref{eq:2}. As $z \in (\C^{\star})^2$, \eqref{eq:2} is equivalent to the equation \begin{equation}~\label{eq:3}
    z_1(\frac{\partial}{\partial z_1}\mathcal{P}(z,\lambda_*)) = 0.
\end{equation}

Recall from Lemma~\ref{lem:newpoly} that $N$ is the Newton polytope of $\mathcal{P}(z,\lambda_*)$, and recall from Corollary~\ref{cor:newpoly} that the Newton polytope $N'$ of $z_1(\frac{\partial}{\partial z_1}\mathcal{P}(z,\lambda_*))$ is contained in $N$.

 By Lemma 4.2 and by the monotonicity property of mixed volume, we have that $MV(N,N') \leq MV(N,N) = 2V(N)$. Calculating the area of $N$, we see that $MV(N,N) = 4q_1q_2$. By Theorem~\ref{thm:Bern}, we conclude that $\#\{k \in [0,1)^2 : \lambda_m(k) = \lambda_*\} \leq \#S_2 \leq \#S_4 \leq MV(N,N) = 4q_1q_2$.
\end{proof}

\section*{Appendix}

\begin{proof}[\bf Proof of Remark 1]
We can first bound the number of solutions in terms of $S_1$. Any solution that satisfies all three of the equations in $S_1$ must also satisfy at least two of them. It follows that we can further bound the cardinality of the level set from above by the number of $z \in \C^2$ such that the following two equations hold,
\begin{equation}
    \mathcal{P}_1(z,\lambda_*) = 0,
\end{equation}
and
\begin{equation}
    \frac{\partial}{\partial z_1}\mathcal{P}_1(z,\lambda_*) = 0.
\end{equation}

Now let's examine $\mathcal{P}(z,\lambda_*)$ as a function in terms of $z_1$, $z_1^{-1}$, $z_2$, and $z_2^{-1}$. By Proposition 2.1, we know that the degree of $\mathcal{P}(z,\lambda_*)$ is at most $q_1$ and the degree of the polynomial $\mathcal{P}_1(z,\lambda_*) = z_1^{q_2}z_2^{q_1}\mathcal{P}(z,\lambda_*)$ is at most $2q_1 + q_2$. Now Theorem 2.4 tells us that $\mathcal P_1(z,\lambda_*)$ is a square-free polynomial with no univariate factors. This means that $\mathcal{P}_1(z,\lambda_*)$ and $\frac{\partial}{\partial z_1}\mathcal{P}_1(z,\lambda_*)$ are polynomials of degree at most $2q_1 + q_2$ and $2q_1 + q_2 - 1$ that share no common factors. Thus, by B\'{e}zout's Theorem, we can see that 

\[\#S_1 \leq (2q_1 + q_2)(2q_1 + q_2 - 1).\]

Finally, by Lemma 2.3, we know that $\#\{k \in [0,1)^2 : \lambda_m(k) = \lambda_*\} \leq \#S_1$, and so we see that

\[\#\{k \in [0,1)^2 : \lambda_m(k) = \lambda_*\} \leq  (2q_1 + q_2)(2q_1 + q_2 - 1).\]
\end{proof}

  \section*{Acknowledgments}
%The authors are very grateful to the anonymous referees for their knowledgeable reports, which helped us to improve our manuscript.
This research was conducted as part of the ongoing Undergraduate Research Program, ``STODO" (Spectral Theory Of Differential Operators), at Texas A\&M University. We are grateful for the support provided by the College of Arts and Sciences Undergraduate Research Program at Texas A\&M University.
 W. Liu was a 2024-2025 Simons fellow.  This work was also partially supported by NSF DMS-2246031, DMS-2052572, DMS-2201005, and  DMS-2052519.

 \section*{Statements and Declarations}
		{\bf Conflict of Interest} 
	The authors  declare no conflicts of interest.
	
	\vspace{0.2in}
	{\bf Data Availability}
	Data sharing is not applicable to this article as no new data were created or analyzed in this study.
\bibliographystyle{abbrv} % abbrv
	\bibliography{main}

\begin{thebibliography}{10}

\bibitem{battig1988toroidal}
D.~B{\"a}ttig.
\newblock {\em A toroidal compactification of the two dimensional {B}loch-manifold}.
\newblock PhD thesis, ETH Zurich, 1988.

\bibitem{bat1}
D.~B\"{a}ttig.
\newblock A directional compactification of the complex {F}ermi surface and isospectrality.
\newblock In {\em S\'{e}minaire sur les \'{E}quations aux {D}\'{e}riv\'{e}es {P}artielles, 1989--1990}, pages Exp. No. IV, 11. \'{E}cole Polytech., Palaiseau, 1990.

\bibitem{batcmh92}
D.~B\"{a}ttig.
\newblock A toroidal compactification of the {F}ermi surface for the discrete {S}chr\"{o}dinger operator.
\newblock {\em Comment. Math. Helv.}, 67(1):1--16, 1992.

\bibitem{bktcm91}
D.~B\"{a}ttig, H.~Kn\"{o}rrer, and E.~Trubowitz.
\newblock A directional compactification of the complex {F}ermi surface.
\newblock {\em Compositio Math.}, 79(2):205--229, 1991.

\bibitem{Bernstein}
D.~N. Bernstein.
\newblock The number of roots of a system of equations.
\newblock {\em Functional Anal. Appl.}, 9(3):183--185, 1975.

\bibitem{col91}
Y.~Colin~de Verdi\`ere.
\newblock Sur les singularit\'{e}s de van {H}ove g\'{e}n\'{e}riques.
\newblock Number~46, pages 99--110. 1991.
\newblock Analyse globale et physique math\'{e}matique (Lyon, 1989).

\bibitem{dksjmp20}
N.~Do, P.~Kuchment, and F.~Sottile.
\newblock Generic properties of dispersion relations for discrete periodic operators.
\newblock {\em J. Math. Phys.}, 61(10):103502, 19, 2020.

\bibitem{ef}
M.~Embree and J.~Fillman.
\newblock Spectra of discrete two-dimensional periodic {S}chr\"{o}dinger operators with small potentials.
\newblock {\em J. Spectr. Theory}, 9(3):1063--1087, 2019.

\bibitem{Ewald}
G.~Ewald.
\newblock {\em Combinatorial convexity and algebraic geometry}, volume 168 of {\em Graduate Texts in Mathematics}.
\newblock Springer-Verlag, New York, 1996.

\bibitem{flmrp}
M.~Faust, W.~Liu, R.~Matos, J.~Robinson, J.~Plute, Y.~Tao, E.~Tran, and C.~Zhuang.
\newblock Floquet isospectrality of the zero potential for discrete periodic {S}chr\"odinger operators.
\newblock {\em J. Math. Phys.}, 65(7):Paper No. 073501, 9, 2024.

\bibitem{fg}
M.~Faust and J.~Lopez-Garcia.
\newblock Irreducibility of the dispersion polynomial for periodic graphs.
\newblock {\em SIAM Journal on Applied Algebra and Geometry (To appear)}, 2024.

\bibitem{FS}
M.~Faust and F.~Sottile.
\newblock Critical points of discrete periodic operators.
\newblock {\em J. Spectr. Theory}, 14(1):1--35, 2024.

\bibitem{fh}
J.~Fillman and R.~Han.
\newblock Discrete {B}ethe-{S}ommerfeld conjecture for triangular, square, and hexagonal lattices.
\newblock {\em J. Spectr. Theory}, 142(1), 2020.

\bibitem{flm22}
J.~Fillman, W.~Liu, and R.~Matos.
\newblock Irreducibility of the {B}loch variety for finite-range {S}chr\"{o}dinger operators.
\newblock {\em J. Funct. Anal.}, 283(10):Paper No. 109670, 22, 2022.

\bibitem{flm23}
J.~Fillman, W.~Liu, and R.~Matos.
\newblock Algebraic properties of the {F}ermi variety for periodic graph operators.
\newblock {\em J. Funct. Anal.}, 286(4):110286, 2024.

\bibitem{fk18}
N.~Filonov and I.~Kachkovskiy.
\newblock On the structure of band edges of 2-dimensional periodic elliptic operators.
\newblock {\em Acta Math.}, 221(1):59--80, 2018.

\bibitem{fk2}
N.~Filonov and I.~Kachkovskiy.
\newblock On spectral bands of discrete periodic operators.
\newblock {\em Comm. Math. Phys.}, 405(21), 2024.

\bibitem{fls}
L.~Fisher, W.~Li, and S.~P. Shipman.
\newblock Reducible {F}ermi surface for multi-layer quantum graphs including stacked graphene.
\newblock {\em Comm. Math. Phys.}, 385(3):1499--1534, 2021.

\bibitem{GKToverview}
D.~Gieseker, H.~Kn\"{o}rrer, and E.~Trubowitz.
\newblock An overview of the geometry of algebraic {F}ermi curves.
\newblock In {\em Algebraic geometry: {S}undance 1988}, volume 116 of {\em Contemp. Math.}, pages 19--46. Amer. Math. Soc., Providence, RI, 1991.

\bibitem{GKTBook}
D.~Gieseker, H.~Kn\"{o}rrer, and E.~Trubowitz.
\newblock {\em The geometry of algebraic {F}ermi curves}, volume~14 of {\em Perspectives in Mathematics}.
\newblock Academic Press, Inc., Boston, MA, 1993.

\bibitem{hj18}
R.~Han and S.~Jitomirskaya.
\newblock Discrete {B}ethe-{S}ommerfeld conjecture.
\newblock {\em Comm. Math. Phys.}, 361(1):205--216, 2018.

\bibitem{Kapi}
T.~Kappeler.
\newblock On isospectral periodic potentials on a discrete lattice. {I}.
\newblock {\em Duke Math. J.}, 57(1):135--150, 1988.

\bibitem{Kapii}
T.~Kappeler.
\newblock On isospectral potentials on a discrete lattice. {II}.
\newblock {\em Adv. in Appl. Math.}, 9(4):428--438, 1988.

\bibitem{kapiii}
T.~Kappeler.
\newblock Isospectral potentials on a discrete lattice. {III}.
\newblock {\em Trans. Amer. Math. Soc.}, 314(2):815--824, 1989.

\bibitem{ktcmh90}
H.~Kn\"{o}rrer and E.~Trubowitz.
\newblock A directional compactification of the complex {B}loch variety.
\newblock {\em Comment. Math. Helv.}, 65(1):114--149, 1990.

\bibitem{kusu01}
P.~Kuchment.
\newblock The mathematics of photonic crystals.
\newblock In {\em Mathematical modeling in optical science}, volume~22 of {\em Frontiers Appl. Math.}, pages 207--272. SIAM, Philadelphia, PA, 2001.

\bibitem{kuc2016}
P.~Kuchment.
\newblock An overview of periodic elliptic operators.
\newblock {\em Bull. Amer. Math. Soc. (N.S.)}, 53(3):343--414, 2016.

\bibitem{kuchment2023analytic}
P.~Kuchment.
\newblock Analytic and algebraic properties of dispersion relations ({B}loch varieties) and {F}ermi surfaces. {W}hat is known and unknown.
\newblock {\em J. Math. Phys.}, 64(11):Paper No. 113504, 11, 2023.

\bibitem{kvcpde20}
P.~Kuchment and B.~Vainberg.
\newblock On absence of embedded eigenvalues for {S}chr\"{o}dinger operators with perturbed periodic potentials.
\newblock {\em Comm. Partial Differential Equations}, 25(9-10):1809--1826, 2000.

\bibitem{kv06cmp}
P.~Kuchment and B.~Vainberg.
\newblock On the structure of eigenfunctions corresponding to embedded eigenvalues of locally perturbed periodic graph operators.
\newblock {\em Comm. Math. Phys.}, 268(3):673--686, 2006.

\bibitem{ls}
W.~Li and S.~P. Shipman.
\newblock Irreducibility of the {F}ermi surface for planar periodic graph operators.
\newblock {\em Lett. Math. Phys.}, 110(9):2543--2572, 2020.

\bibitem{liu1}
W.~Liu.
\newblock Irreducibility of the {F}ermi variety for discrete periodic {S}chr\"{o}dinger operators and embedded eigenvalues.
\newblock {\em Geom. Funct. Anal.}, 32(1):1--30, 2022.

\bibitem{liujmp22}
W.~Liu.
\newblock Topics on {F}ermi varieties of discrete periodic {S}chr\"{o}dinger operators.
\newblock {\em J. Math. Phys.}, 63(2):Paper No. 023503, 13, 2022.

\bibitem{liu2d}
W.~Liu.
\newblock Fermi isospectrality of discrete periodic {S}chr\"{o}dinger operators with separable potentials on {$\Bbb Z^2$}.
\newblock {\em Comm. Math. Phys.}, 399(2):1139--1149, 2023.

\bibitem{liujde}
W.~Liu.
\newblock Floquet isospectrality for periodic graph operators.
\newblock {\em J. Differential Equations}, 374:642--653, 2023.

\bibitem{liuborg}
W.~Liu.
\newblock Proof of geometric {B}org's theorem in arbitrary dimensions.
\newblock {\em arXiv preprint arXiv:2306.16412}, 2023.

\bibitem{liu2022bloch}
W.~Liu.
\newblock Bloch varieties and quantum ergodicity for periodic graph operators.
\newblock {\em J. Anal. Math.}, 153(2), 2024.

\bibitem{lmt}
W.~Liu, R.~Matos, and J.~N. Treuer.
\newblock Sharp decay rate for eigenfunctions of perturbed periodic schr\"odinger operators, 2024.

\bibitem{ms22}
T.~Mckenzie and M.~Sabri.
\newblock Quantum ergodicity for periodic graphs.
\newblock {\em Comm. Math. Phys.}, 403(3):1477--1509, Nov. 2023.

\bibitem{sy23}
M.~Sabri and P.~Youssef.
\newblock Flat bands of periodic graphs.
\newblock {\em J. Math. Phys.}, 64(9):Paper No. 092101, 23, 2023.

\bibitem{shi1}
S.~P. Shipman.
\newblock Eigenfunctions of unbounded support for embedded eigenvalues of locally perturbed periodic graph operators.
\newblock {\em Comm. Math. Phys.}, 332(2):605--626, 2014.

\bibitem{shjst20}
S.~P. Shipman.
\newblock Reducible {F}ermi surfaces for non-symmetric bilayer quantum-graph operators.
\newblock {\em J. Spectr. Theory}, 10(1):33--72, 2020.

\end{thebibliography}
	
\end{document}